\documentclass[10pt, twoside, reqno]{amsart}

\usepackage[english]{babel}

\usepackage{amsmath,amssymb}

\newcommand{\N}{{\mathbb N}}

\newcommand{\upto}{,\ldots ,}

\newtheorem{thm}{Theorem}

\newtheorem{lem}{Lemma}

\theoremstyle{definition}

\newtheorem*{remark}{Remark}

\numberwithin{equation}{section}

\usepackage{stmaryrd}

\usepackage[arrow, matrix, curve]{xy}

\usepackage{enumerate}

\usepackage[pdfpagemode=None,
			baseurl={https://www-m11.ma.tum.de},
           	pdftitle={Reimers 2019},
           	pdfauthor={Fabian Reimers}]{hyperref}

\begin{document}

\title{Separating Invariants for Two Copies of the Natural $S_n$-action}

\author{Fabian Reimers}

\address{Technische Universi\"at M\"unchen, Zentrum Mathematik - M11, 
Boltzmannstr.~3, 85748 Garching, Germany}

\email{reimers@ma.tum.de}

\date{May 10, 2019}



\begin{abstract}
This note provides a set of separating invariants for the ring of vector invariants $K[V^2]^{S_n}$ of two copies of the natural $S_n$-representation $V = K^n$ over a field of characteristic 0. This set is much smaller than generating sets of $K[V^2]^{S_n}$.

For $n \leq 4$ we show that this set is minimal with respect to inclusion among all separating sets.
\end{abstract}

\maketitle

\section*{Introduction}

Let $K$ be a field of characteristic 0. Let $G = S_n$ be the symmetric group and $V = K^n$ its natural representation over $K$. The $K$-algebra of symmetric polynomials $K[V]^{S_n} = K[x_1 \upto x_n]^{S_n}$ is generated by the first $n$ power sums
\[ x_1 + \ldots + x_n, \quad x_1^2 + \ldots + x_n^2, \quad \ldots \quad, \quad x_1^n + \ldots + x_n^n.\]
The action of $S_n$ on $V$ extends to a diagonal action of $S_n$ on $V^m$ for $m \geq 2$. The ring of vector invariants $K[V^m]^{S_n}$ is not a polynomial ring anymore. Its elements are called \emph{multisymmetric polynomials} and they have been classically studied (see Weyl \cite{weyl1939classical}). If we write the variables of $K[V^m]$ as $x_{ij}$ with $i = 1 \upto n$, $j = 1 \upto m$, then $K[V^m]^{S_n}$ is minimally generated by the following polarizations of the above power sums: \vspace{-1mm}
\[ f_{k_1 \upto k_m} := \sum_{i=1}^n x_{i1}^{k_1} \cdot \ldots \cdot x_{im}^{k_m} \quad \text{ with $k_i \in \N$ such that $k_1 + \ldots + k_m \leq n$},\]
see Domokos \cite[Theorem 2.5 \& Remark 2.6]{domokos2009vector}.
Asymptotically, this minimal generating set has size $\mathcal{O}(n^m)$. But by the general upper bound for separating invariants (see Derksen and Kemper \cite[Chapter 2.4]{derksen2002computationalv2}), a subset of $K[V^m]^{S_n}$ of size $2 nm + 1 = \mathcal{O}(nm)$ which separates all $S_n$-orbits in $V^m$ exists. A lower bound of $(n+1)m-1$ for the size of a separating set was achieved by Dufresne and Jeffries \cite[Theorem 3.4]{dufresne2015separating}.

However, to the best of my knowledge no explicit small separating sets for multisymmetric polynomials are known. Using the concept of ``cheap polarizations'' Draisma et al. \cite[Corollary 2.12]{draisma2008polarization} give a separating set of size $\mathcal{O}(n^2 m)$. 

This note considers the case $m = 2$. Here the two sets of variables are denoted by $x_1 \upto x_n$ and $y_1 \upto y_n$. Then $K[V^2]^{S_n}$ is minimally generated by the set of power sums 
\begin{equation}\label{EquationFjk} M = \{ \underbrace{\sum_{i=1}^n x_i^j y_i^k}_{=: \, f_{j,k}} \mid j,k \in \N \text{ such that } j+k \leq n \text{ and } jk \neq 0 \}. \end{equation}
We use the notation $f_{j,k} = \sum_{i=1}^n x_i^j y_i^k$ and $M$ as in (\ref{EquationFjk}) throughout this paper. Observe that
\[ |M| = \left(\sum_{j=0}^n (n-j+1)\right) - 1= \frac{1}{2} (n^2 +3n).\]
The main result of this paper is the separating set $S$ given in the following theorem, which is much smaller than $M$.
\bigskip 

\begin{thm}\label{TheoremMain}
Let $I$ be the set of double indices 
\[ I = \{ (j,k) \in \N^2 \mid 0 \leq j \leq n,\,\, 0 \leq k \leq \lfloor \frac{n}{j+1} \rfloor,\,\, jk \neq 0 \} .\]
Then $S := \{ f_{j,k} \mid (j,k) \in I \}$ is a separating subset of the ring of bisymmetric poly-\\[2mm] nomials $K[V^2]^{S_n}$. 
\end{thm}

\bigskip 

\begin{remark}(a) Note that $$|S| = \left( \sum_{j=0}^{n} (\lfloor \frac{n}{j+1}\rfloor + 1) \right) - 1 = n + \sum_{j=1}^n \lfloor \frac{n}{j}\rfloor.$$
The sum $D(n) := \sum_{j=1}^n \lfloor \frac{n}{j}\rfloor$ is the so-called Divisor summatory function. Asymptotically, we have \vspace{-1mm}$$|S| = \mathcal{O}(n \log(n))\vspace{1mm}$$ in contrast to $|M| = \mathcal{O}(n^2)$.
For example, for $n = 100$ we have $|M| = 5150$ and $|S| = 582$.

(b) The invariants from $S$ are homogeneous, while the general upper bound of $2nm + 1$ for separating invariants can typically only be achieved with inhomogeneous invariants.

(c) The set $I$ is not symmetric in $(j,k)$. For example, for $n = 3$ we have $f_{2,1} \in S$, but $f_{1,2} \notin S$.
\end{remark}

\medskip

We give the proof of this theorem in Section 1. In Section 2 we analyze the cases $n = 2$, $3$, $4$ more closely and show that $S$ is indeed a \emph{minimal} separating set.

\smallskip

\section{Proof of the main theorem}\label{section}

We use induction on $n$ to proof Theorem \ref{TheoremMain}. For $n = 1$ we have $S = M = \{ x_1,y_1\}$ which is clearly separating (and generating) for the trivial action on $K[V^2] = K[x_1,y_1]$.

Now let $n \geq 2$ and take two points $p$, $q \in V^2$ such that $f(p) = f(q)$ for all $f \in S$. We need to show that $p$ and $q$ lie in the same $S_n$-orbit. 

For this let $p = (a_1 \upto a_n, b_1 \upto b_n)$. The \glqq{}original\grqq{} power sums $f_{j,0} = \sum_{i=1}^n x_i^j$ with $j = 1 \upto n$ are elements of $S$. From $f_{j,0}(p) = f_{j,0}(q)$ for all $j$ we get that the $x_i$-coordinates of $p$ and $q$ are the same upto a permutation. So there exists an element $\sigma \in S_n$ with $\sigma q = (a_1 \upto a_n, c_1 \upto  c_n)$. 
Since we only want to show that $p$ and $q$ are equivalent under the $S_n$-action, we may replace $q$ by $\sigma q$. So now we may assume that 
\[ p = (a_1 \upto a_n, b_1 \upto b_n) \quad \text{and} \quad q = (a_1 \upto a_n, c_1 \upto c_n).\]

To exhibit the main idea of the following proof let us first do the much easier special case that all $a_i$ are distinct. The elements $f_{j,1}$ with $j = 0 \upto n-1$ are in $S$, hence for all these $j$ we have
\[ \sum_{i=1}^n a_i^j b_i = f_{j,1}(p) = f_{j,1}(q) =  \sum_{i=1}^n a_i^j c_i.\]
Using a Vandermonde matrix these equations can be written as
\[ \begin{pmatrix}1 &\ldots & 1\\ a_1 & \ldots & a_n \\ \vdots \\ a_1^{n-1} & \ldots & a_n^{n-1}\end{pmatrix} \cdot \begin{pmatrix}b_1\\\vdots\\b_n\end{pmatrix} = \begin{pmatrix}1 &\ldots & 1\\ a_1 & \ldots & a_n \\ \vdots \\ a_1^{n-1} & \ldots & a_n^{n-1}\end{pmatrix} \cdot \begin{pmatrix}c_1\\\vdots\\c_n\end{pmatrix}.\] Now if we assume that the $a_i$ are pairwise distinct, we can conclude $b_i = c_i$ for all $i$ and hence $p = q$. So indeed they lie in the same $S_n$-orbit.

For the general case let $r := |\{ a_1 \upto a_n \}|$ be the number of distinct $a_i$'s and let $$\{ \lambda_1 \upto \lambda_r \} = \{a_1 \upto a_n\}.$$ We can use a permutation $\tau \in S_n$ to rearrange the $a_i$'s such that 
\[ \tau (a_1 \upto a_n) = (\underbrace{\lambda_1 \upto \lambda_1}_{n_1}, \underbrace{\lambda_2 \upto \lambda_2}_{n_2} \upto \underbrace{\lambda_r \upto \lambda_r}_{n_r}),\]
where $(n_1 \upto n_r)$ is a partition of $n$, i.e., $n = n_1 + \ldots + n_r$ and $n_1 \geq n_2 \geq \ldots \geq n_r$.

By replacing $p$ and $q$ with $\tau p$ and $\tau q$, respectively, we may now assume that 
\begin{align}\label{FormOfpq} p &= (a_1 \upto a_n,b_1 \upto b_n) = (\underbrace{\lambda_1 \upto \lambda_1}_{n_1}, \underbrace{\lambda_2 \upto \lambda_2}_{n_2} \upto \underbrace{\lambda_r \upto \lambda_r}_{n_r}, b_1 \upto b_n),\notag \\
q &= (a_1 \upto a_n,c_1 \upto c_n) = (\underbrace{\lambda_1 \upto \lambda_1}_{n_1}, \underbrace{\lambda_2 \upto \lambda_2}_{n_2} \upto \underbrace{\lambda_r \upto \lambda_r}_{n_r}, c_1 \upto c_n),
\end{align}
where the $\lambda_i$ are pairwise distinct and $n_1 \geq n_2 \geq \ldots \geq n_r$ holds. 

Observe that $r n_r \leq n$, i.e., $n_r \leq \frac{n}{r}$. For every $j = 0 \upto r-1$ and $k = 1 \upto n_r$ we have $(j,k) \in I$, since:
\begin{align*} j \leq r-1 \leq n-1, \\ k \leq n_r \leq \frac{n}{r} \leq \frac{n}{j+1}. \end{align*} 

So $f_{j,k} \in S$ for $j = 0 \upto r-1$ and $k = 1 \upto n_r$. Hence for these $j$ and $k$ we have $f_{j,k}(p) = f_{j,k}(q)$ by assumption.
It is 
\begin{align*} f_{j,k}(p) = \sum_{i=1}^n a_i^j b_i^k &= \lambda_1^j (b_1^k + \ldots + b_{n_1}^k) + \ldots + \lambda_r^j (b_{n - n_r + 1}^k + \ldots + b_{n}^k )\\
&= \sum_{i=1}^r \lambda_i^j \left(\sum_{l = 1}^{n_i} b_{n_1 + \ldots + n_{i-1} + l}^k \right),
\end{align*}
and similarly 
\[ f_{j,k}(q) = \sum_{i=1}^r \lambda_i^j \left(\sum_{l = 1}^{n_i} c_{n_1 + \ldots + n_{i-1} + l}^k \right).\]
For a fixed $k \in \{ 1 \upto n_r \}$ this holds for every $j = 0 \upto r-1$. So we conclude that the vectors
\begin{equation}\label{VectorsOfbAndc} \begin{pmatrix}b_1^k + \ldots + b_{n_1}^k \\ b_{n_1+1}^k + \ldots + b_{n_1+n_2}^k \\ \vdots \\ b_{n - n_r + 1}^k + \ldots + b_{n}^k\end{pmatrix} \quad \text{and} \quad \begin{pmatrix}c_1^k + \ldots + c_{n_1}^k \\ c_{n_1+1}^k + \ldots + c_{n_1+n_2}^k \\ \vdots \\ c_{n - n_r + 1}^k + \ldots + c_{n}^k\end{pmatrix} \in K^r\end{equation}
have the same image when multiplied from the right to the Vandermonde matrix
\[ \begin{pmatrix}1 &\ldots & 1\\ \lambda_1 & \ldots &\lambda_r \\ \vdots & & \vdots \\ \lambda_1^{r-1} & \ldots & \lambda_r^{r-1}\end{pmatrix}.\]
Since the $\lambda_i$ are pairwise distinct, the vectors in (\ref{VectorsOfbAndc}) must be equal. In the last component this means
\[   b_{n - n_r + 1}^k + \ldots + b_{n}^k = c_{n - n_r + 1}^k + \ldots + c_{n}^k.\]
But this holds or every $k = 1 \upto n_r$, so we can conclude that the $n_r$ elements \linebreak$b_{n - n_r + 1},\, b_{n - n_r + 2} \upto b_{n}$ are equal to the $n_r$ elements $c_{n - n_r + 1},\, c_{n - n_r + 2} \upto c_{n}$ upto a permutation $\pi \in S_{n_r}$. 

We read $\pi$ as an element of $S_n$ which permutes only the \emph{last} $n_r$ positions. So applying $\pi$ to a point in $V^2 = K^n\times K^n$ can only change the last $n_r$ of the $x_i$- and the last $n_r$ of the $y_i$-coordinates of this point. But observe that because of (\ref{FormOfpq}) applying $\pi$ to $q$ has no affect on the last $n_r$ of the $x_i$-coordinates of $q$. So after replacing $q$ with $\pi q$ we may now assume that the last $n_r$ of the $y_i$-coordinates of $p$ and $q$ are equal as well, while $p$ and $q$ are still in the form of (\ref{FormOfpq}).

Now we project onto the first $n - n_r$ of both the $x_i$- and $y_i$-coordinates. This projection $K^n \times K^n \to K^{n-n_r} \times K^{n-n_r}$ maps $p$ and $q$ to
\[ p' = (a_1 \upto a_{n-n_r},b_1 \upto b_{n-n_r})\] and \[q' =  (a_1 \upto a_{n-n_r},c_1 \upto c_{n-n_r}),\] respectively. For $n-n_r$ instead of $n$ we have a corresponding set 
\[ I' = \{ (j,k) \in \N^2 \mid 0 \leq j \leq n-n_r,\,\, 0 \leq k \leq \lfloor \frac{n-n_r}{j+1} \rfloor,\,\, jk \neq 0 \}.\]
of double indices, which leads to a corresponding set $S' = \{ f_{j,k}' \mid (j,k) \in I' \}$ of polynomials $$f_{j,k}' := \sum_{i=1}^{n-n_r} x_i^j y_i^k$$ in the ring of bisymmetric polynomials $$K[x_1 \upto x_{n-n_r},y_1 \upto y_{n-n_r}]^{S_{n-n_r}}.$$ By induction assumption, $S'$ is separating. 

For every double index $(j,k) \in I'$ we have $(j,k) \in I$. So for every $(j,k) \in I'$ we get 
\begin{align*} f_{j,k}'(p') = \sum_{i = 1}^{n-n_r} a_i^j b_i^k &= f_{j,k}(p) - \sum_{i=n - n_r+1}^{n} a_i^j b_i^k \\ &= f_{j,k}(q) - \sum_{i=n-n_r+1}^{n} a_i^j c_i^k = f_{j,k}'(q'),\end{align*} 
where we have used that the last $n_r$ coordinates of $p$ and $q$ are equal.
Since $S'$ is separating, there exists a permutation $\varphi \in S_{n - n_r}$ such that $p' = \varphi q'$. We read $\varphi$ naturally as an element of $S_n$ fixing the last $n_r$ positions. Since $p$ and $q$ were already equal at the last $n_r$ of the $x_i$- and the last $n_r$ of the $y_i$-coordinates, we conclude that $p = \varphi q$. \hfill $\qed$

\section{Minimality in low-dimensional cases}\label{section}

In this section we consider the question of minimality (w.\,r.\,t.\, inclusion) of the set $S$ from Theorem \ref{TheoremMain} among all separating subsets of $K[V \times V]^{S_n}$. For $n \leq 4$ we show that $S$ is indeed a minimal separating set.

First we note in the following lemma that the power sums in $x_1 \upto x_n$ (and similarly the power sums in $y_1 \upto y_n$) cannot be left out from $S$. 

\smallskip 

\begin{lem}\label{LemmaUsualPowerSums}
Let $S$ be the set from Theorem \ref{TheoremMain}.
For all $j \in \{ 1 \upto n\}$ the set $S \setminus \{ f_{j,0} \}$ is not separating. Similarly, for all $k \in \{ 1 \upto n\}$ the set $S \setminus \{ f_{0,k} \}$ is not separating.
\end{lem}

\begin{proof}
We can view $T = \{ f_{j,0} \mid j = 1 \upto n \}$ as a subset of the ring of symmetric polynomials $K[x_1 \upto x_n]^{S_n}$. Here a minimal separating set has size $n$, so for all $j \in \{ 1 \upto n \}$ the set $T \setminus \{ f_{j,0} \}$ is not separating. Hence there exist $a = (a_1 \upto a_n)$ and $a' = (a_1' \upto a_n') \in K^n$ such that $g(a) = g(a')$ for all $g \in T \setminus \{ f_{j,0} \}$. But then $(a,0)$, $(a',0)$ in $K^{2n}$ cannot be separated by $S \setminus \{ f_{j,0} \}$. 
The second statement is analogous.
\end{proof}

Next we note in the following lemma another polynomial that cannot be left out from $S$.

\smallskip 

\begin{lem}\label{LemmaAnotherPolynomial}
Let $S$ be the set from Theorem \ref{TheoremMain}. Assume that $n$ is even and set $r := \frac n 2$. Then the set $S \setminus \{ f_{1,r}\}$ is not separating.
\end{lem}

\begin{proof}
In $K[x_1 \upto x_r]^{S_r}$ the set 
\[ \{ \sum_{i=1}^{r} x_i^k \mid k = 1 \upto r - 1 \}\] is not separating (since the last power sum is missing). So there exist
$b = (b_1 \upto b_r)$ and $c = (c_1 \upto c_r)$ which are not in the same $S_r$-orbit, but satisfy
\[ \sum_{i=1}^{r} b_i^k = \sum_{i=1}^{r} c_i^k \quad \text{for } k = 1 \upto r - 1.\] 
Consider 
\[ p = (\underbrace{1 \upto 1}_{r},\, \underbrace{2 \upto 2}_{r},\, b_1 \upto b_r,\,c_1 \upto c_r) \in K^{n} \times K^n \]
and 
\[ q = (\underbrace{1 \upto 1}_{r},\, \underbrace{2 \upto 2}_{r},\, c_1 \upto c_r,\,b_1 \upto b_r) \in K^{n} \times K^n. \]
We have 
\[ f_{j,k}(p) = \sum_{i=1}^{r} b_i^k + \sum_{i=1}^{r} 2^j c_i^k, \quad \text{ and } \quad f_{j,k}(q) = \sum_{i=1}^{r} c_i^k + \sum_{i=1}^{r} 2^j b_i^k,\]
hence
\[ f_{j,k}(p) - f_{j,k}(q) = (2^j-1) \sum_{i=1}^{r} (b_i^k - c_i^k).\]
So if $j = 0$ or if $k \leq r - 1$ we have $f_{j,k}(p) - f_{j,k}(q) = 0$. Looking at the definition of $S$ in Theorem \ref{TheoremMain} we see that for all $f_{j,k} \in S$ with $(j,k) \neq (1,r)$ we have $j = 0$ or $k \leq r-1$, hence $f_{j,k}(p) = f_{j,k}(q)$. But $p$ and $q$ are not in the same $S_n$-orbit, so $S \setminus \{ f_{1,r}\}$ is not separating.
\end{proof}

\smallskip 

\begin{thm}\label{TheoremMinimal}
For $n \leq 4$ the set $S$ from Theorem \ref{TheoremMain} is a minimal separating set with respect to inclusion.
\end{thm}

\begin{proof}
Case $n = 2$: Here we have
 \[ S = M = \{ \underbrace{x_1 + x_2}_{=\,f_{1,0}},\,\,\, \underbrace{x_1^2 + x_2^2}_{=\,f_{2,0}} ,\,\,\, \underbrace{y_1 + y_2}_{=\,f_{0,1}},\,\,\, \underbrace{y_1^2 + y_2^2}_{=\,f_{0,2}},\,\,\, \underbrace{x_1y_1 + x_2y_2}_{=\,f_{1,1}} \}.\] By Lemma \ref{LemmaUsualPowerSums}, $S \setminus \{ f_{j,k} \}$ with $j = 0$ or $k = 0$ is not separating. By Lemma \ref{LemmaAnotherPolynomial}, $S \setminus \{ f_{1,1} \}$ is not separating.

Case $n = 3$: Here 
\[ S = \{ f_{j,0} \mid j = 1,2,3 \} \cup \{ f_{0,k} \mid k = 1,2,3 \} \cup \{ f_{1,1},\,f_{2,1} \}. \]
By Lemma 1 there are only two cases left to check. 

For example 
\[ p = (1,\,2,\,3,\,1,\,0,\,2),\quad q = (1,\,2,\,3,\,0,\,2,\,1) \]
show that $S \setminus \{ f_{2,1} \}$ is not separating,
while
\[ p = (1,\,2,\,3,\,5,\,0,\,8),\quad q = (1,\,2,\,3,\,0,\,8,\,5)\]
show that $S \setminus \{ f_{1,1} \}$ is not separating.

Case $n = 4$: Here 
\[ S = \{ f_{j,0} \mid j = 1,2,3,4 \} \cup \{ f_{0,k} \mid k = 1,2,3,4 \} \cup \{ f_{1,1},\, f_{2,1},\,f_{3,1},\, f_{1,2} \} .\]
Again by Lemma 1 there are only four cases left to check. Lemma 2 shows that $S \setminus \{ f_{1,2} \}$ is not separating.

For example 
\[ p = (1,\,2,\,3,\,4,\,0,\,3,\,1,\,4) ,\quad q = (1,\,2,\,3,\,4,\,1,\,0,\,4,\,3)  \]
show that $S \setminus \{ f_{3,1} \}$ is not separating.

For example 
\[ p = (3,\,1,\,0,\,-5,\,6,\,0,\,10,\,1) ,\quad q = (3,\,1,\,0,\,-5,\,1,\,10,\,6,\,0) \]
show that $S \setminus \{ f_{2,1} \}$ is not separating.

For example 
\[ p = (-1,\,0,\,1,\,2,\,1,\,7,\,-3,\,9) ,\quad q = (-1,\,0,\,1,\,2,\,-3,\,1,\,9,\,7) \]
show that $S \setminus \{ f_{1,1} \}$ is not separating.
\end{proof}

\bigskip


\bibliographystyle{myplain}
\bibliography{literature}

\end{document}